\documentclass{amsart}
\usepackage{amssymb, xspace}
\usepackage[all]{xy}
\usepackage{enumerate}

\newif\ifpdf
\ifx\pdfoutput\undefined
\pdffalse 
\else
\pdfoutput=1 
\pdftrue
\fi

\ifpdf
\usepackage[pdftex]{graphicx}
\else
\usepackage{graphicx}
\fi

\theoremstyle{plain}
\newtheorem{theorem}{Theorem}[section]

\newtheorem{lemma}[theorem]{Lemma}
\newtheorem{proposition}[theorem]{Proposition}

\theoremstyle{definition}
\newtheorem{definition}[theorem]{Definition}

\theoremstyle{remark}

\numberwithin{equation}{section}

\newcommand{\projM}{\rho_M}

\begin{document}

\ifpdf
\DeclareGraphicsExtensions{.pdf, .jpg, .tif}
\else
\DeclareGraphicsExtensions{.eps, .jpg}
\fi

\title{Erratic boundary images of CAT(0) geodesics under $G$-equivariant maps}
\author{D. Staley}
\address {Department of Mathematics\\
    Rutgers University,
    Piscataway, NJ 08854-8019}
\email{staley@math.rutgers.edu}

\date{\today}

\begin{abstract}
We show that, given any finite dimensional, connected, compact metric space $Z$, there exists a group $G$ acting geometrically on two CAT(0) spaces $X$ and $Y$, a $G$-equivariant quasi-isometry $f\colon X\rightarrow Y$, and a geodesic ray $c$ in $X$ such that the closure of $f(c)$, intersected with $\partial Y$, is homeomorphic to $Z$.  This characterizes all homeomorphism types of ``geodesic boundary images'' that arise in this manner.
\end{abstract}

\maketitle

\large

\section{Introduction}

CAT(0) spaces are geodesic metric spaces which satisfy a certain notion of nonpositive curvature.  For an exact definition and overview of CAT(0) spaces the reader is referred to \cite{BH99}.  A CAT(0) space $X$ has a natural boundary $\partial X$, whose points consist of geodesic rays in $X$ emanating from a chosen basepoint $p$.  This definition of boundary is independent of choice of $p$.  There is a natural topology on $X \cup \partial X$ which is called the {\em cone topology}.    Unfortunately, the boundary of a CAT(0) space is not a quasi-isometry invariant:  In~\cite{CK00} Croke and Kleiner construct examples of a group $G$ which acts properly and cocompactly by isometries on different spaces whose boundaries are not homeomorphic.

In \cite{Be96} Bestvina adapted the proof of Chapman's complement theorem to show that if a group acts on spaces with different boundaries, then the boundaries are shape equivalent.  He then asked whether all such boundaries are CE-equivalent.  Much progress has been made in classifying boundary types of various groups (for example see \cite{BR96} and \cite{CMpp2}), but Bestvina's question remains open.

Given a CAT(0) space $X$ and a subset $Q \subset X$, we can view $Q$ as a subset of $X \cup \partial X$ under the cone topology.  Letting $\overline{Q}$ denote the closure of $Q$, we call $\overline{Q} \cap \partial X$ the {\em boundary limit} of~$Q$.

A group action on a metric space is called {\em geometric} if the action is proper, cocompact, and by isometries.  A fundamental theorem of geometric group theory states that if a group $G$ acts geometrically on a proper metric space $X$, then $X$ is quasi-isometric to $G$ with the word metric.  In particular, if $G$ acts geometrically on two CAT(0) spaces $X$ and $Y$, then there is a $G$-equivariant quasi-isometry from $X$ to $Y$.

One may take the image of a geodesic ray in $X$ under this quasi-isometry, and examine its boundary limit in $Y$.  It is not hard to see that such boundary limits must be compact, connected, and finite dimensional.  The main result of this chapter is that these are the only restrictions we can put on the homeomorphism types of such boundary limits:

\begin{theorem}\label{T:Main}
Let $Z$ be any connected, compact subset of Euclidean space.  Then there is a group $G$ which acts geometrically on two CAT(0) spaces $X$ and $Y$, a $G$-equivariant quasi-isometry $f:X\rightarrow Y$, and a geodesic $c$ in $X$, such that the boundary limit of $f(c)$ is homeomorphic to $Z$.
\end{theorem}

In fact in our construction $X$ and $Y$ will be the same space, albeit with a different action of $G$.  Our construction of $X$, $G$, and $f$ depends only on $n$, thus all $Z \subset \mathbb{R}^n$ arise as boundary limits obtained from one particular quasi-isometry.

The construction in this chapter is heavily influenced by the constructions of Croke and Kleiner in \cite{CK00} and its generalizations by C. Mooney in \cite{CMpp}.  In fact, one may view the space $X$ as a higher-dimensional analogue of the ``blocks'' used in their construction.

If $c$ is a path in a space, we will occasionally fail to distinguish between $c$ and its image.

For a point $p$ in a metric space and $r > 0$, we will denote by $B_r(p)$ the open ball of radius $r$ centered at $p$, and for a set $Q$ we will use $N_r(Q)$ to mean the $r$-neighborhood of the set $Q$, i.e., $\bigcup_{p\in Q} B_r(p)$.

The chapter is organized as follows:  Section \ref{S:construct} constructs the group $G$, its actions on the space $X$, and the quasi-isometry $f$.  Section \ref{S:GeoIms} describes the behavior of certain geodesics under the map $f$.  Section \ref{S:Vwalk} defines the concept of a {\em directed path} in a simplex, then proves the main theorem, deferring the proofs of some technical lemmas to the final section.  Section \ref{S:Remarks} gives some remarks on the construction.  Finally, Section \ref{S:Lemmas} proves the lemmas required for the arguments in Section \ref{S:Vwalk}.

\section{Constructing $X$, $G$, and $f$}\label{S:construct}

We begin by constructing our CAT(0) space $X$, and the group $G$ which will act on it geometrically.  Throughout the chapter we will assume a fixed integer $n$.

Consider $(n+1)$-dimensional Euclidean space.  Consider the group of isometries on this space generated by translations by $e_0,...,e_n$, the standard basis vectors.  The quotient of Euclidean space by this action is $T^{n+1}$, the standard $(n+1)$-torus.  Take $n$ disjoint copies of $T^{n+1}$, and identify the images of the subspaces spanned by $e_1,...,e_n$ in each torus.  Call the resulting space $\bar{X}$, and the identified image of the origin $\bar{p}$.  Denote its universal cover by $X$.  Choose a lift $p \in X$ of $\bar{p}$.  The reader may easily verify the following properties:

\begin{proposition} \begin{enumerate}

\item $\bar{X}$ is a space of nonpositive curvature.

\item $\pi_1(\bar{X}) = F_n \times \mathbb{Z}^n$, where $F_n$ denotes the free group on $n$ generators.

\item $X$ is a CAT(0) space, and is isometric to $T \times \mathbb{R}^n$, where $T$ is a regular tree of degree $2n$.

\item $p$ can be chosen so that $p = (q,\overrightarrow{0}) \in T \times \mathbb{R}^n = X$, where $q$ is a vertex of $T$.

\end{enumerate} \end{proposition}

We then have that the group $G = F_n \times \mathbb{Z}^n$ acts geometrically on the CAT(0) space $X$.  The inverse image of each $T^{n+1}$ in $X$ is a disjoint union of Euclidean $(n+1)$-spaces, exactly one of which contains $p$ for each torus.  Call these spaces $E_1,...,E_n \subset X$.  Denote by $g_1,...,g_n$ the generators of the $F_n$ factor of $G$.  We may chose these generators so that the subgroup $\langle g_i \rangle \times \mathbb{Z}^n$, acts geometrically on $E_i$, via translations by $e_0,...,e_n$, such that $g_i$ translates by $e_0$ and the standard generators of $\mathbb{Z}^n$ translate by $e_1,...,e_n$.  Indeed, this allows us to assume that the last $n$ coordinates of each $E_i$ correspond exactly to the $\mathbb{R}^n$ factor of $X$.  We will use the symbol~$\cdot$ for this group action.

Define a group automorphism $\Phi \colon G \rightarrow G$ in the following way:  Choosing generators $e_1,...,e_n$ of $\mathbb{Z}^n$ in the standard way, let $\Phi\bigl((1,e_i)\bigr) =  (1,e_i)$.  Let $\Phi\bigl((g_i, \overrightarrow{0})\bigr) = (g_i, e_i)$.  The reader may easily check that $\Phi$ is an automorphism.  This allows us to define a new action of $G$ on $X$, $\circ$, by $g \circ x$ = $\Phi(g) \cdot x$.

\begin{proposition}\label{P:circGeo}
The action $\circ$ is geometric.
\end{proposition}

\begin{proof}
Since $\Phi$ is an automorphism, $\circ$ has the same orbits as $\cdot$, and thus is proper and cocompact.  Indeed, since each transformation $\Phi(g) \cdot G$ is an isometry, each $g$ acts via $\circ$ by an isometry as well.  Thus, $\circ$ is a geometric action.
\end{proof}

For each $i = 1,...,n$, we define a linear transformation $f_i \colon E_i \rightarrow E_i$, which sends $e_0$ to $e_0 + e_i$ and is the identity on $e_1,...,e_n$.  Note that on each $E_i$, $f_i$ is $G$-equivariant in the sense that $f_i(g \cdot x) = g \circ f_i(x)$ for all $g \in \langle g_i \rangle \times \mathbb{Z}^n, x \in E_i$.

\begin{proposition}\label{P:fEqui}

\begin{enumerate}

\item $X$ is a union of flats of the form $g \cdot E_i$, with $g \in F_n \times \bigl\{\overrightarrow{0}\bigr\} \subset G$.

\item For $i \neq j$ and all $x \in E_i \cap E_j$, $f_i(x) = f_j(x) = x$.

\item There is a unique function $f \colon X \rightarrow X$ such that $f\mid_{E_i} = f_i$ for each $i = 1,...,n$, and such that $f$ is $G$-equivariant in the sense that $f(g \cdot x) = g \circ f(x)$ for all $g \in G, x \in X$.

\end{enumerate} \end{proposition}

\begin{proof}  $X$ is of the form $T \times \mathbb{R}^n$, and so we can write $p$ in the form $(p_T, \overrightarrow{0})$, where $p_T$ is a vertex of $T$.  Since the $F_n$ factor of $G$ acts on $T$ in the standard way, any vertex in $T$ can be written as $g \cdot p_T$ for some $g \in F_n$.  The edges coming out of $p_T$ are all contained in one of the $E_i$, and represent either the $e_0$ or $-e_0$ direction.  Thus, any point $q$ on any edge can be expressed as $g \cdot v$ with $v = (a_0,0,0,...,0) \in E_i$.  Thus, any point $\left(q,(a_1,...,a_n)\right) \in T \times \mathbb{R}^n$ can be written as $g \cdot v'$ with $v' = (a_0,a_1,...,a_n) \in E_i$, proving (1).

The intersection of $E_i$ and $E_j$ for $i \neq j$ is the connected component of the inverse image of the common $n$-torus in $\bar{X}$ which contains $p$.  This is isometric to Euclidean $n$-space, and can be described as the convex hull of the orbit of $p$ under the action of the $\mathbb{Z}^n$ factor of $G$.  This is the space spanned by $e_1,...,e_n$ in the coordinates of both $E_i$ and $E_j$, and both $f_i$ and $f_j$ are the identity map on this space, which proves (2).

To see (3), we can simply define $f(g \cdot x) = g \circ f_i(x)$ for $x \in E_i$.  (1) and (2) guarantee that this function is well-defined on all of $X$.  (1), together with the equivariance condition, guarantees the uniqueness of $f$.

\end{proof}

Defining $f$ as the unique function from Proposition \ref{P:fEqui} (3), we record here some of its properties.

\begin{proposition}\label{P:projcommutes}
For all $(q,v) \in X$, $f\bigl((q,v)\bigr) = (q, w)$ for some $w \in \mathbb{R}^n$, that is, $f$ is the identity on the $T$ factor of $X$.
\end{proposition}

\begin{proof}
The proposition follows from the definition of $f$, since the $T$ coordinate of $x \in E_i$ depends only on the $e_0$ coordinate of $x$ in $E_i$, which is unchanged by $f_i$.  For points not in any of the $E_i$, since $f(g \cdot x) = g \circ f_i(x)$, we need only check that~$\circ$ and~$\cdot$ act identically on the $T$ factor of $X$.  Since $g \circ x = \Phi(g) \cdot x$ and $\Phi$ doesn't affect the $F_n$ factor of $G$, the result follows as the subgroup $\{1\} \times \mathbb{Z}^n \subset G$ acts trivially on the $T$ factor of $X$.
\end{proof}

We will put coordinates on the flats $g \cdot E_i$, for $g \in F_n \times \{\overrightarrow{0}\}$, however, $g \cdot E_i$ and $g' \cdot E_i$ may be the same flat for distinct $g, g'$.  Since $E_i$ is the convex hull of the orbit of the basepoint under the subgroup $\langle g_i \rangle \times \mathbb{Z}^n$ for some generator $g_i$, we have $g \cdot E_i = g' \cdot E_i$ if and only if $g' = gg_i^n$ for some $n \in \mathbb{Z}$.  Thus, for any such flat, we can always choose a $g$ of minimal length so that our flat is of the form $g \cdot E_i$, and we put coordinates on such a flat by translating the coordinates of $E_i$ via the isometry induced by this minimal $g$.

\begin{proposition}\label{P:affine}
For $g \in G$ and $i \in \{1,...,n\}$, $f(g \cdot E_i) = g \cdot E_i$, and $f\mid_{g\cdot E_i}$ is an affine transformation of $g\cdot E_i$.
\end{proposition}

\begin{proof}
Choose $g$ so that the origin in $g \cdot E_i$ is $g \cdot p$.  Let $A\colon E_i \rightarrow E_i$ be defined by $A(x) = g^{-1} \cdot \left(g \circ x \right)$.  Since $g \circ x = f(g \cdot x)$, it suffices to prove that $A$ is affine.

Since $\circ$ and $\cdot$ have the same effect on the $T$ component of $X$, we immediately see that $A$ is the identity on the $e_0$ component of $E_i$.  Let $u$ be the vector given by $g \circ p - g \cdot p$, taken in the coordinates of $g \cdot E_i$.  Then for $x \in E_i$, we have $g \circ x = g \cdot f_i(x) + u$ since the coordinates of $g \cdot E_i$ are just translated coordinates of $E_i$.  Letting $A_u$ be translation by $-u$, we then have that $A_u(A(x)) = f_i(x)$.

Thus, after composing $A$ with a translation we have that $A$ is linear, and so $A$ and therefore $f\mid_{g\cdot E_i}$ is an affine transformation.


\end{proof}

The remainder of the chapter is devoted to proving the following, which in light of Propositions \ref{P:circGeo} and \ref{P:fEqui}(3) implies the main theorem:

\begin{theorem}\label{T:impliesMain} Let $X$ and $f$ be defined as above.  Given any compact, connected set $Z \subset \mathbb{R}^{n-1}$, there is a geodesic ray $c:[0,\infty)\rightarrow X$ such that the boundary limit of $f(c)$ is homeomorphic to $Z$.
\end{theorem}

\section{Geodesic images under $f$}\label{S:GeoIms}

This section will use the notation defined in section \ref{S:construct}.

Consider a geodesic ray $c \colon [0, \infty)$ in the $2n$-regular tree $T$, with the property $c(0) = p_T$.  Since edges in $T$ have length 1, $c(t)$ is a vertex precisely when $t$ is an integer.  So, under the standard action of $F_n$ on $T$, $c(1) = g_{i_1}(p)$ for some generator $g_{i_1} \in F_n$, $c(2) = g_{i_1}g_{i_2}(p)$ for some second generator $g_{i_2}$, and in general, we have

\[
c(k) = w_k(p),
\]

\noindent where $w_k \in F_n$ is represented by a word of length $k$ in the generators and their inverses.  For the remainder of the section, we will assume that $c$ is chosen such that $w_k$ is a {\em positive} word in the $g_i$, that is, can be written using only the generators and never their inverses.  $w_k$ is a proper initial substring of $w_{k+1}$, and so there is a unique infinite sequence of generators such that every $w_k$ is an initial subsequence.  We define a sequence of integers $\{I_k\}$ such that this infinite sequence is $g_{I_1}, g_{I_2}, g_{I_3},...$.

For the remainder of the chapter, we will consider $T$ to be a subset of $X$ by identifying it with $T \times \{\overrightarrow{0}\}$.  Since $c \subset T$, $c$ lies in the subspace $c \times \mathbb{R}^n$.

\begin{proposition} \label{imInC}
$f(c) \subset c \times \mathbb{R}^n$.
\end{proposition}

\begin{proof}
This is an easy consequence of Proposition \ref{P:projcommutes}.
\end{proof}

The space $c \times \mathbb{R}^n$ is a Euclidean half-space of dimension $n+1$, and is a closed, convex subset of $X$.  We put coordinates on this half space by identifying the point $(c(t), a_1,...,a_n)$ with $(t,a_1,...,a_n) \in \mathbb{R}^{n+1}$.

\begin{proposition}\label{P:spaceType}
For each integer $k \geq 0$, the space $c\mid_{[k,k+1]} \times \mathbb{R}^n$ lies in exactly one of the flats $g \cdot E_i$.  In this space, the coordinates on $c \times \mathbb{R}^n$ and on $g \cdot E_i$ differ by a constant vector $v$, which is an integer multiple of $e_0$.
\end{proposition}

\begin{proof}
The segment $c\mid_{[k,k+1]}$ is an edge of the tree $T$, and each edge of $T$ is contained in exactly one flat $g \cdot E_i$.  But since the flat containing a point depends only on the value of the $T$ factor of $X$, we have that $c\mid_{[k,k+1]} \times \mathbb{R}^n$ is contained in exactly one $g \cdot E_i$.

The last $n$ coordinates in $c \times \mathbb{R}^n$ are the coordinates from the $\mathbb{R}^n$ factor of $X$.  But $g \cdot E_i$ has coordinates translated from $E_i$ via an element of $F_n \times \{\overrightarrow{0}\}$, which doesn't change the $\mathbb{R}^n$ factor.  Since the last $n$ coordinates of $E_i$ are just the coordinates from the $\mathbb{R}^n$ factor of $X$, we see that the last $n$ coordinates of $g \cdot E_i$ and $c \times \mathbb{R}^n$ agree on their intersection.

If $x$ is a point in $c\mid_{[k,k+1]} \times \mathbb{R}^n$, then the $e_0$ coordinate in both $c \times \mathbb{R}^n$ and $g \cdot E_i$ differs by a constant from $d\left(\pi_T(x), c(k)\right)$, where $\pi_T$ is projection onto the $T$ factor.  Thus the two coordinates differ by a constant from each other.  Furthermore, since the origin (under both coordinate systems) is a vertex of $T$, they must differ by an integer constant, proving the proposition.
\end{proof}

\begin{definition} Let $V = \{v_1,...,v_m\}$ be a finite set of vectors in a Euclidean half-space $E$.  A path $c:[0,\infty)\rightarrow E$ or $c:[0,a]\rightarrow E$ is a {\em walk in $E$ over $V$} if $a \in \mathbb{Z}$ and there is a sequence $v_{j_1}, v_{j_2}, v_{j_3},...$ of vectors in $V$ such that for each integer $k$ with $[k,k+1]$ in the domain of $c$, $c(k+1) = c(k) + v_{j_k}$, and $c\mid_{[k, k+1]}$ maps linearly to the line segment connecting $c(k)$ and $c(k) + v_{j_k}$.
\end{definition}

Note that a walk is determined by a finite or infinite sequence of vectors in $V$ together with a starting point $c(0)$.

For $i = 1,...,n$, let $v_i = e_i + e_0$.  Recall that the sequence of integers $\{I_k\}$ was chosen so that $c(k) = g_{I_1}...g_{I_k}(p)$.

\begin{lemma} \label{P:cWalk} $f(c)$ is a walk in $c \times \mathbb{R}^n$ over the set $\{v_1,...,v_n\}$ starting at the origin.  The sequence of vectors corresponding to this walk is $v_{I_1}, v_{I_2}, v_{I_3},...$.
\end{lemma}

\begin{proof}
By Proposition \ref{P:spaceType}, the space $c\mid_{[k,k+1]} \times \mathbb{R}^n$ lies in some $g \cdot E_i$, and the coordinates on the two spaces coincide up to translation by an integer multiple of $e_0$.  In particular, $c\left([k,k+1]\right)$ is a line segment in the coordinates of $g \cdot E_i$, and since $f$ is affine on this space by Proposition \ref{P:affine}, $f\bigl(c\left([k,k+1]\right)\bigr)$ is a line segment in both coordinate systems.

To show that $f(c)$ is a walk, we thus need only to show that the vector $f(c(k+1)) - f(c(k))$ is one of $v_1,...,v_n$. This can be done in either coordinate system since they only differ by a translation.  Since $c(k+1) - c(k) = e_0$ in both coordinate systems, we then have $f(c(k+1)) - f(c(k)) = e_0 + e_i = v_i$, proving the proposition.
\end{proof}

\section{V-walks and boundary limits}\label{S:Vwalk}

In this section we will prove Theorem \ref{T:impliesMain}, deferring some technical arguments to the next section.  Recall that the {\em Hausdorff distance} $d_H(S,S')$ between two subsets $S$ and $S'$ of a metric space is the infimum of the set $\{\epsilon > 0 \mid S \subset N_\epsilon(S'), S' \subset N_\epsilon(S)\}$.

Consider a Euclidean half-space $E = \{(x_1,...,x_n) \in \mathbb{R}^n \mid x_0 \geq 0\}$.  $E$ is a CAT(0) space, and its boundary is a closed half-sphere $M$, which we give the angle metric (this is in fact the Tits metric).  Since each point in the boundary is represented by a unique ray emanating from $\overrightarrow{0}$, we will occasionally fail to distinguish between the boundary point and the ray. It will also be convenient to think of points in $E$ as vectors.

Given any $x \neq \overrightarrow{0} \in E$, we can draw the geodesic segment to $\overrightarrow{0}$ and then extend to a ray emanating from $\overrightarrow{0}$.  This gives a map $\projM \colon E - \{\overrightarrow{0}\} \rightarrow M$.  We will call the image of a point under this map the {\em projection of $x$ to $M$}.  Note that, for vectors $x$ and $u$, we have that $d\bigl(\projM(x), \projM(u)\bigr)$ is exactly the angle between $x$ and $u$ at the origin.

Let $V$ be any set of linearly independent vectors in $E$.  Let $W = \projM(V)$, and let $L \subset M$ be the spherical simplex obtained by taking the convex hull of $W$.  The result we will need to prove Theorem \ref{T:impliesMain} is the following:

\begin{theorem} \label{T:TimpW} Suppose $Z \subset L$ is connected and compact.  Then there is a walk in $E$ over $V$, starting from the origin, whose boundary limit is $Z$.
\end{theorem}

Nonsingular linear transformations induce homeomorphisms of Euclidean space and its boundary under the cone topology.  Thus, by applying a change-of-basis transformation and passing to a subspace we may assume that $V = \{e_1,...,e_n\}$.

\begin{definition}
A path $\gamma:[a,a+1] \rightarrow L$ is a {\em $W$-directed segment} if $\gamma$ is injective and its image is an initial subsegment of a geodesic segment $[\gamma(a),w]$ for some $w \in W$.
\end{definition}

\begin{definition}
A path $\gamma:[0,\infty) \rightarrow L$ or $\gamma:[0,a] \rightarrow L$ is a {\em $W$-directed path} if $a \in \mathbb{Z}$, and for each positive integer $k$ in the domain of $\gamma$, $\gamma \mid_{[k-1,k]}$ is either constant or a $W$-directed segment.
\end{definition}

The proof of Theorem \ref{T:TimpW} will rely on the following two lemmas.  The proofs of these lemmas are rather technical and are deferred to the final section.

\begin{lemma}\label{L:directApprox}
Let $\gamma:[0,a] \rightarrow L$ be any path, and let $\epsilon > 0$.  Then there is is a $W$-directed path $\gamma':[0,a']\rightarrow L$ such that $d_H(\gamma,\gamma') \leq \epsilon$.  $\gamma'$ may be chosen so that $\gamma'(0)$ is any point in $B_\epsilon(\gamma(0))$, and so that $\gamma'(a') \in B_\frac{\epsilon}{2}(\gamma(a))$.
\end{lemma}

\begin{lemma} \label{L:closeWalk}
Let $\gamma:[0,a] \rightarrow L$ be a $W$-directed path, and let $\epsilon_1,\epsilon_2 > 0$.  Then there is an $R>0$ such that, for any $x \in E$ with $\lVert x \rVert \geq R$ and $d\bigl(\projM(x), \gamma(0)\bigr) < \frac{\epsilon_1}{2}$, there is a finite walk $c$ in $E$ over $V$, starting at $x$, such that $d_H\bigl(\projM(c),\gamma\bigr) \leq \epsilon_1$.  The walk $c$ may be chosen to have arbitrarily long length, and so that the projection of its ending point is within $\epsilon_2$ of $\gamma(a)$.
\end{lemma}

We now prove some basic propositions.

\begin{proposition}\label{P:boundType}
Let $c:[0,\infty) \rightarrow E$ be a path such that $\lim_{t\rightarrow \infty}\lVert c(t) \rVert = \infty$.  Then the boundary limit of $c$ is
\[
\left\{\lim_{k\rightarrow \infty}\projM\bigl(c(a_k)\bigr) \Big\vert a_k \in \mathbb{R}^+, \lim_{k\rightarrow \infty} a_k = \infty \right\}.
\]
\end{proposition}

\begin{proof}
Let $q$ be a point in the boundary limit of $c$.  Then there is a sequence of points $c(a_1), c(a_2),...$ which converge to $q$ in the cone topology.  Recall that a basis of open sets around $q$ are sets of the form $N(q,R,\epsilon) = \{x \in E \mid d(x, \overrightarrow{0}) > R, d(\rho_R(x),\rho_R(q)) < \epsilon\}$, where $\rho_R$ is radial projection to the closed ball of radius $R$ centered at $\overrightarrow{0}$.

For any $R$, all but finitely many of the $c(a_k)$ have $\lVert c(a_k) \rVert > R$.  It follows that $\lim_{k\rightarrow \infty}a_k = \infty$.  Since $d\bigl(\rho_R(c(a_k)),\rho_R(q)\bigr) < \epsilon$, the angle between $c(a_k)$ and $q$ is less than $\sin^{-1}(\frac{\epsilon}{R})$, thus the angles between $c(a_k)$ and $q$ approach zero.  Therefore, $\lim_{k\rightarrow \infty}\projM\left(c(a_k)\right) = q$.

Conversely, suppose $q = \lim_{k\rightarrow\infty}\projM\left(c(a_k)\right)$, with $a_k \rightarrow \infty$.  Then the angle between $q$ and $c(a_k)$ approaches 0 as $k \rightarrow \infty$, and thus for any $R$ and $\epsilon$, we can find $K$ such that $d\bigl(\rho_R(c(a_k)),\rho_R(q)\bigr) < \epsilon$ for all $k > K$.  Since the $a_k$ approach $\infty$ we can also choose $K$ so that $\lVert c(a_k) \rVert > R$ for $k > K$.  Thus $q$ is the limit of the points $c(a_k)$ in the cone topology, and the proposition is proved.
\end{proof}

\begin{lemma}\label{L:matchBound}
Suppose $\gamma:[0,\infty) \rightarrow L$ is any path.  Then there is a walk $c:[0,\infty)\rightarrow E$ over $V = \{e_1,...,e_n\}$ such that $c(0) = \overrightarrow{0}$ and the boundary limit of $c$ is $$\bigcap_{T>0}\overline{\gamma\left([T,\infty)\right)}.$$
\end{lemma}

\begin{proof}
Express $\gamma$ as a concatenation of paths $\gamma_1\colon [0,1]\rightarrow L, \gamma_2\colon[1,2]\rightarrow L$, etc.  By Lemma \ref{L:directApprox}, for each $\gamma_i$, we can choose a directed path $\gamma_i'$ such that $d_H(\gamma_i,\gamma_i') \leq \frac{1}{2^i}$, and we can choose these $\gamma_i'$ so that the starting point of $\gamma_i'$ is the ending point of $\gamma_{i-1}'$.

Letting $\epsilon_i = \frac{1}{2^i}$, let $R_i$ be the associated value of $R$ needed to approximate $\gamma_i$ with Lemma \ref{L:closeWalk}.  Choose any point $v \in E$ with positive integer coordinates such that $\lVert v \rVert > R_1$ and $d\bigl(\projM(v),\gamma_1(0)\bigr) < \frac{\epsilon_1}{2}$.  Then by Lemma \ref{L:closeWalk}, there is a walk $c_1$ starting at $v$ such that $d_H(\projM(c_1),\gamma_1') \leq \epsilon_1$.  We can choose $c_1$ to have arbitrary length, in particular we can choose $c_1$ so that its ending point has distance greater than $R_2$ from the origin.  We can also choose $c_1$ so that its ending point, when projected to $M$, has distance less than $\frac{\epsilon_2}{2}$ from $\gamma_2(0)$.

Lemma \ref{L:closeWalk} then says there is a walk $c_2$, starting from the endpoint of $c_1$, ending at a point of distance at least $R_3$ from the origin, and such  that $d_H(\projM(c_2),\gamma_2') \leq \epsilon_2$.  We may choose $c_2$ so that its ending point, when projected to $M$, has distance less than $\frac{\epsilon_3}{2}$ from $\gamma_3(0)$.

Continuing in this way, we construct a sequence of walks $c_1,c_2,c_3,...$, each ending where the last started, such that $d_H(c_i,\gamma_i') \leq \epsilon_i$.  Concatenating all these walks together, and concatenating any walk from the origin to $v$ at the beginning gives us a walk $c \colon [0,\infty) \rightarrow E$.

By Proposition \ref{P:boundType}, the boundary limit of $c$ consists precisely of limits of sequences of the form $\projM(c(a_k))$, with $a_k \rightarrow \infty$.  Given such a sequence, for each $k$, we can choose a point $q_k$ on some $\gamma_{i_k}'$ such that $i_k \rightarrow \infty$ and $d(\projM(c(a_k)), q_k) \rightarrow 0$ as $k \rightarrow \infty$.  We can then choose $b_1,b_2,\dots \in \mathbb{R}$ such that $b_k \rightarrow \infty$, and so that $d(\gamma(b_k),q_k) \rightarrow 0$.  Thus $d(\projM(c(a_k)), \gamma(b_k)) \rightarrow 0$ as $k \rightarrow \infty$.

The limit of a sequence $\gamma(b_k)$ with $b_k \rightarrow \infty$ is precisely a point in $\bigcap_{T>0}\overline{\gamma\left([T,\infty)\right)}$.  Thus we have that the boundary limit is contained in this intersection.  But for any sequence $\gamma(b_k)$, we can choose $q_k \in \gamma_{i_k}'$ with $i_k \rightarrow \infty$ and $d(\gamma(b_k), q_k) \rightarrow 0$.  Then we can choose a sequence $a_k \rightarrow \infty$ such that $d(\projM(c(a_k)), q_k) \rightarrow 0$, and so the sequence $\projM(c(a_k))$ converges to the same point as $\gamma(b_k)$.  This shows the reverse inclusion, proving the proposition.
\end{proof}

We are now in a position to prove Theorem \ref{T:TimpW}.

\begin{proof}[Proof of Theorem \ref{T:TimpW}]
Fix any point $q \in Z$.  For each positive integer $k$, let $S_k$ be a finite subset of $Z$ such that $Z \subset N_\frac{1}{k}(S_k)$.  Since $Z$ is a connected subset of a simplex, any open neighborhood of $Z$ is path-connected.  Thus we can choose paths $\gamma_k \colon [0,1]\rightarrow N_\frac{1}{k}(Z)$ such that $\gamma_k(0) = \gamma_k(1) = q$ and such that $S_k \subset \gamma_k$.  Concatenating the paths $\gamma_k$ gives a path $\gamma\colon [0,\infty) \rightarrow L$.

Note that any point of $Z$ lies in the closure of $\gamma$, since the $S_k$ get arbitrarily close to every point of $Z$.  Indeed, any point of $Z$ lies in $\overline{\gamma\left([T,\infty)\right)}$ for any $T \in \mathbb{R}$.  For $k \in \mathbb{Z}$, we have $\gamma\left([k,\infty)\right) \subset N_{\frac{1}{k}}(Z)$, and thus $Z = \bigcap_{T>0}\overline{\gamma\left([T,\infty)\right)}$.

By Lemma \ref{L:matchBound}, there is a walk over $V$ starting at the origin which has $\bigcap_{T>0}\overline{\gamma\left([T,\infty)\right)} = Z$ as its boundary limit.
\end{proof}

\begin{proof}[Proof of Theorem \ref{T:impliesMain}]
For $i = 1,...,n$, let $v_i = e_0 + e_i$.  Let $V = \{v_1,...,v_n\}$.  Let $L$ be the spherical $(n-1)$-simplex obtained by taking the convex hull of $\projM(V)$.  Embed $Z$ into $L$.  By Theorem \ref{T:TimpW}, there is a walk $c$ over $V$ starting at the origin whose boundary limit is $Z$.  Let $v_{I_1},v_{I_2},...$ be the sequence of vectors in $V$ associated to the walk $c$.

Take the geodesic ray in $T$ which starts at $p$ and passes through $g_{I_1}\cdot p, g_{I_1}g_{I_2} \cdot p$, etc., and call this ray $c'$.  $f(c') \subset c' \times \mathbb{R}^n$ by Proposition \ref{P:projcommutes}.  By Lemma \ref{P:cWalk}, $f(c') = c$ in the coordinates of $c' \times \mathbb{R}^n$.  Since $c' \times \mathbb{R}^n$ is a closed, convex subset of $X$, its boundary is embedded in $\partial X$.  Thus, the boundary limit of $f(c')$ is the boundary limit of $c$ and is homeomorphic to $Z$.
\end{proof}

As mentioned, Theorem \ref{T:impliesMain} implies the main theorem.

\section{Remarks}\label{S:Remarks}

It is interesting to note that the construction of the group $G$, the space $X$, and the map $f$ depended only on the dimension of the space $Z$.  This gives us that, for a fixed $n$, every compact connected subspace of $\mathbb{R}^{n-1}$ occurs as a boundary limit of $f(c)$ for some geodesic $c$ in $X$.  While these boundary limits have diverse homeomorphism types, all of their inclusions into $\partial X$ are nullhomotopic.  This leaves open whether the boundary limit of the image of a geodesic ray under a $G$-equivariant quasi-isometry is always homotopic to a point in the boundary.

In the proof of the main theorem, only positive geodesics were considered for simplicity's sake.  However, the argument of \ref{P:cWalk} extends to geodesics in $T \subset X$ which do not represent positive words.  Whenever an inverse of a generator $g_i^{-1}$ occurs, the geodesic still passes through some $g \cdot E_i$, but $c(k+1)-c(k)$ is now $-e_0$ in the coordinates of $g \cdot E_i$.  This means that a direction in the coordinates of $c \times \mathbb{R}^n$ is the reflection across the $e_0 = 0$ hyperplane of the direction in the coordinates of $g \cdot E_i$.  So $f$ sends the vector $c(k+1) - c(k)$ to the reflection of $f_i(-e_0)$ across the hyperplane $e_0 = 0$.  Since $f(-e_0) = -e_0 - e_i$, its reflection is $e_0 - e_i$.  Calling such a vector $v'_i$, we see that the image of such a geodesic is then just a walk in $\{v_1,...,v_n\} \cup \{v'_1,...,v'_n\}$.

The geodesics we have constructed are all of the form $c' \times \{0\} \subset T \times \mathbb{R}^n = X$, for some geodesic $c'$ in the tree $T$.  One may also ask what happens to the images of geodesics $c$ which do not stay in the tree $T$.  We answer this with the following proposition:

\begin{proposition} Let $c$ be any geodesic in $X = T \times \mathbb{R}^n$.  Let $c' = \pi_T(c)$, the projection of $c$ to the $T$ factor.  If $c'$ is not constant, then up to reparameterization $c'$ is a geodesic, and the boundary limit of $f(c)$ is homeomorphic to the boundary limit of $f(c')$.
\end{proposition}

\begin{proof}  $c'$ is a path in a tree which does not retrace itself, thus is a geodesic after reparameterization.  Note that $c$ must be contained in the Euclidean half-space $c' \times \mathbb{R}^n$, and as a geodesic ray it corresponds to some vector in this space.  This vector must have a strictly positive $e_0$ coordinate, otherwise $c'$ would be constant.  Let $v$ be the vector pointing in this direction, scaled so that the $e_0$ component of $v$ is 1.

The same argument as in the proof of Proposition \ref{P:cWalk} then shows that, in the coordinates of $c' \times \mathbb{R}^n$, $f(c)$ is a walk in the vectors $f_1(v),...,f_n(v)$.  Letting $v = u + e_0$, we then have that $f_i(v) = u + e_i + e_0$.  This means the vectors $f_i(v)$ are still linearly independent (one easily checks that adding $u + e_0$ to these vectors produces a basis of $\mathbb{R}^{n+1}$).  Thus, $f(c)$ is a walk in $f_1(v),...,f_n(v)$, and the sequence of these vectors is exactly the sequence of vectors $v_1,...,v_n$ taken by $f(c')$ with $v_i$ replaced by $f_i(v)$.

Let $L$ be a change-of-basis linear transformation of Euclidean space which takes $\{v_1,...,v_n\}$ to $\{f_1(v),...,f_n(n)\}$.  This extends to a homeomorphism of Euclidean space together with its boundary sphere.  $L$ also takes $f(c')$ to $f(c)$.  Thus, the boundary image of $f(c)$ is the image under $L$ of the boundary image of $f(c')$, and so the two are homeomorphic.

\end{proof}

\section{Directed paths in a simplex}\label{S:Lemmas}
In this section we will prove Lemmas \ref{L:directApprox} and \ref{L:closeWalk}.

\begin{proposition} It suffices to prove Lemma \ref{L:directApprox} for a standard Euclidean simplex $K$.
\end{proposition}

\begin{proof} The projection from the standard Euclidean simplex to the corresponding spherical simplex sends vertices to vertices and geodesics to geodesics.  Thus a directed path on $K$ will project to a directed path on the spherical simplex $L$.

Furthermore, since each point of $K$ is at least distance $\frac{1}{\sqrt{n}}$ from the origin, points on $K$ of distance $d$ apart have angle no more than $2\sin^{-1}\left(2d\sqrt{n}\right) \leq 2\pi d\sqrt{n}$.  Thus, projection induces a bounded distortion of distances, and so if we choose a directed path approximating a path within $\frac{\epsilon}{2\pi\sqrt{n}}$ in $K$, then the projection will approximate the projected path within $\epsilon$ in $L$.
\end{proof}

In light of the above proposition, we will proceed to prove Lemma \ref{L:directApprox} for a Euclidean simplex $K$.  We will still refer to the vertex set of $K$ as $W$.

\begin{lemma} \label{L:lineApprox}
Let $q \in K$.  Let $K'$ be a face of $K$, and let $q'$ lie in the interior of $K'$.  Then for any $\epsilon > 0$, there is a $W$-directed path $\gamma$ starting at $q$, ending in $B_\epsilon(q')$, such that $\gamma \subset N_\epsilon\bigl([q, q']\bigr)$.
\end{lemma}

\begin{proof}The proof will be by induction on the dimension of $K'$.  If the dimension is zero then $K'$ is a vertex, and so the geodesic from $q$ to $q'$ is already $W$-directed.

Let the dimension $d$ of $K'$ be at least 1, and choose a vertex $v$.  The other vertices span a $(d-1)$--simplex $K''$.  Extend the geodesic segment $[v,q']$ to $K''$, and denote by $q''$ the point where this segment intersects $K''$.  Let $P$ denote the 2-dimensional plane containing $q,q'$, and $v$.  Let $\pi_P \colon K \rightarrow P$ denote orthogonal projection to $P$.

Suppose $i$ is an integer with $i>1$.  Since the dimension of $K''$ is $d-1$, the inductive hypothesis implies that there is a $W$-directed path $\beta$ from $q$, ending within $\frac{\epsilon}{2^i}$ of $q''$, staying within $\frac{\epsilon}{2^i}$ of the segment $[q,q'']$.  If $d(q'', q') \leq \frac{\epsilon}{2}$ then by setting $\gamma = \beta$ we are done, otherwise let $x$ be the point on $[q,q'']$ which is of distance $\frac{\epsilon}{2}$ from $[q,q']$.  Let $a \in \mathbb{R}$ be such that $d(\beta(a), x) < \frac{\epsilon}{2^i}$.

Let $\beta' = \beta \mid_{[0,a]}$.  Note that $\beta'$ is still $W$-directed.

Since $q'$ lies in the interior of $K'$, we have that $v$ and $q''$ lie on opposite sides of $[q, q']$ in $P$.  By the triangle inequality, this means that $\pi_P(\beta'(a))$ lies on the same side of $[q,q']$ as $q''$.  Thus, the segment $[\beta'(a), v]$, when projected to $P$, crosses the segment $[q,q']$.  Append to $\beta'$ the $v$-directed segment that terminates at the point $y$ which projects to this intersection.

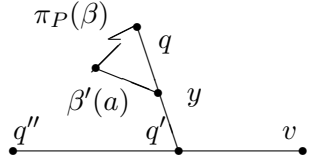
\begin{figure}[h]

\begin{center}

\setlength{\unitlength}{.55cm}
\begin{picture}(7,3)
\put(0,0){\circle*{0.2}}
\put(0,.3){$q''$}
\put(4,0){\circle*{0.2}}
\put(7,0){\circle*{0.2}}
\put(3.2,.3){$q'$}
\put(6.5,0.3){$v$}
\put(0,0){\line(1,0){7}}
\put(3,3){\circle*{.2}}
\put(3.5,2.5){$q$}
\put(2.3,2.7){\line(2,1){.7}}
\put(2.3,2.7){\line(1,0){.3}}
\put(1.9,1.9){\line(1,1){.7}}
\put(.5,3.1){$\pi_P(\beta)$}

\put(2,2){\circle*{.2}}
\put(1.3,1){$\beta'(a)$}
\put(2,2){\line(5,-2){1.5}}
\put(3.5,1.4){\circle*{.2}}
\put(3,3){\line(1,-3){1}}
\put(4.2,1.2){$y$}

\end{picture}
\end{center}

\caption{The construction projected to the plane $P$.}

\end{figure}

The segment $[\beta'(a),y]$ starts at most $\frac{\epsilon}{2^i}$ from the plane $P$ and travels towards a point on $P$, and thus ends within $\frac{\epsilon}{2^i}$ of $P$.  In particular $d(y,[q,q']) < \frac{\epsilon}{2^i}$.  Thus if $d(y,[v,q'']) < \frac{\epsilon}{2}$, we are done.  Otherwise, this segment has a length of at least $\frac{\epsilon}{4}$, since its projection starts at least of distance $\frac{\epsilon}{4}$ from $[q,q']$ and ends at distance 0.

Consider the line containing $q'$ and $v$ to be the $x$-axis in $P$.  Then the line segment $\pi_P\left([\beta(a),y]\right)$ starts at least distance $\frac{\epsilon}{2}$ from the $x$-axis, and it moves towards a point on the $x$-axis no more than distance $\mathrm{diam}(K) = \sqrt{2}$ away.  If we choose our orientation so that the $x$ values of the segment are increasing, we then have that its slope is less than $-\frac{\epsilon}{2\sqrt{2}}$.  Since it is of length at least $\frac{\epsilon}{4}$, this gives a bound on $d\left(\pi_P(\beta'(a)),[v,q'']\right) - d\left(\pi_P(y),[v,q'']\right)$ which is independent of $i$.

Since $\beta$ stayed within $\frac{\epsilon}{2^i}$ of the geodesic $[q,q'']$, if $i$ is sufficiently large we have $d\left(\pi_P(\beta(a)),[v,q'']\right) < d(q,[v,q''])$.  Thus, assuming we choose $i$ sufficiently large, we have that $d(q,[v,q'']) - d(y,[v,q''])$ is bounded away from zero.

We can now continually repeat this process, replacing $q$ by $y$ and increasing $i$ sufficiently each time.  We concatenate the results to obtain a $W$-directed path $\gamma$ starting from $q$.  Since at each stage we move no farther than $\frac{\epsilon}{2^i}$ from the geodesic to $q'$ at the previous stage, we stay within $\epsilon$ of $[q,q']$ at every point in this process.  Further, since each step ends closer to $[v,q'']$ by an amount bounded away from zero, the process eventually terminates within $\frac{\epsilon}{2}$ of $[v,q'']$, and thus within $\epsilon$ of $q'$.
\end{proof}

In particular, setting $K' = K$ in the above Lemma allows us to approximate any line segment by a $W$-directed path.

\begin{proof}[Proof of Lemma \ref{L:directApprox}] Lemma \ref{L:lineApprox} implies the existence of a $W$-directed path from any $x \in B_\epsilon(\gamma(0))$ which ends in $B_\frac{\epsilon}{2}(\gamma(0))$ and stays within $\epsilon - d\bigl(\gamma(0),x)$ of the segment $[x,\gamma(0)]$.  By beginning with such a path, we can assume that our starting point lies in $B_\frac{\epsilon}{2}(\gamma(0))$.

Approximate $\gamma$ by a piecewise-linear path $\beta$ such that $d_H(\gamma,\beta) < \frac{\epsilon}{2}$.  By Lemma \ref{L:lineApprox} we can find $W$-directed paths for each linear segment, each starting where the previous ended, ending within $\frac{\epsilon}{2}$ of the next linear segment, and staying within $\frac{\epsilon}{2}$ of the segments connecting their starting and ending points.  Thus each stays within $\epsilon$ of the corresponding linear segment of $\beta$.

Concatenating these $W$-directed paths gives a single $W$-directed path $\gamma'$ with the desired properties.
\end{proof}

We will need the following proposition for the proof of Lemma \ref{L:closeWalk}:

\begin{proposition}\label{P:projectSeg}
Let $x \in E$ be any nonzero vector with nonnegative coordinates, and let $e_i$ be any elementary basis vector.  Let $u = x + e_i$.  Then the segment $[x,u]$ projects to a $W$-directed segment on $M$, and $d(\projM(x), \projM(u)) \leq \sin^{-1}\left(\frac{1}{\lVert x \rVert}\right)$.
\end{proposition}

\begin{proof}
The segment $[x,u]$ consists of vectors of the form $x + te_i$, for real values $t$.  All such vectors lie in the plane spanned by $x$ and $e_i$.  The projection of this plane to $M$ is precisely the geodesic containing $\projM(x)$ and $\projM(e_i)$.  But $\projM(e_i) \in W$ is a vertex $w_i$ of the simplex $L$.  It is clear that $d(\projM(x + te_i), w) \geq d(\projM(x + t'e_i),w)$ for $0 \leq t \leq t' \leq 1$.  Thus the segment $[x,u]$ projects to a subsegment of $[\projM(x),w]$, which is a $W$-directed segment.

Since $v$ and $e_i$ have fixed lengths, $\projM([x,u])$ has maximum length if $u$ is orthogonal to $e_i$.  In this case, its length is the angle between $x$ and $u$, which is $\sin^{-1}\left(\frac{\lVert e_i \rVert}{\lVert x \rVert}\right)$.  Since $\lVert e_i \rVert = 1$, this gives that $d(\projM(x), \projM(u)) \leq \sin^{-1}\left( \frac{1}{ \lVert x \rVert } \right)$.
\end{proof}

\begin{proposition}\label{P:closeVert}
Given any vector $v$ and basis vector $e_i$,
\[
\lim_{k\rightarrow \infty} d\bigl(\projM(v + ke_i),\projM(e_i)\bigr) = 0.
\]
\end{proposition}
\begin{proof}
By the same argument as above, the angle between $ke_i$ and $ke_i + v$ is no more than $\sin^{-1}\left(\frac{\lVert v \rVert}{\lVert ke_i \rVert}\right)$.  Since $\projM(ke_i) = \projM(e_i)$ and $\lVert ke_i \rVert \rightarrow \infty$ the proposition follows.
\end{proof}

In particular, Propositions \ref{P:projectSeg} and \ref{P:closeVert} show that we can add $k$ copies of $e_i$ to any vector $v$, and if $k$ is large enough then $\projM(ke_i+v)$ gets arbitrarily close to the vertex $\projM(e_i)$ of $L$.  Additionally, the linear path from $\projM(v)$ to $\projM(ke_i+v)$ projects to a $\projM(e_i)$-directed segment.

\begin{lemma} \label{L:getsClose}  Suppose $a, a'$, and $w$ are colinear in the half-sphere $M$, and $a'$ is on the geodesic segment from $a$ to $w$.  Further suppose $b \in M$ with $d(b, w) < d(a', w)$, and all distances between these points are $\leq \frac{\pi}{2}$.  Then $d(b,a') < d(b,a)$.
\end{lemma}

\begin{figure}[h]
\begin{center}
\includegraphics[width = 2.5in]{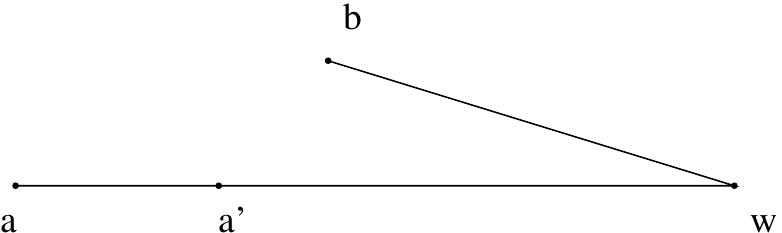}
\end{center}
\caption{}
\end{figure}

\begin{proof}
This is an exercise in spherical geometry.  For this proof, we will use the convention $\overline{ab}$ to denote $d(a,b)$.  Since all distances here are $\leq \frac{\pi}{2}$, we have that $\sin$ and $\tan$ are strictly increasing functions, while $\cos$ is a strictly decreasing function.  Let $C$ be the angle (on the sphere) formed by the geodesic segments $[b,w]$ and $[a,w]$.  Since $d(a,w) > d(a',w) > d(b,w)$ and $\cos(C) < 1$, we have

\[
\tan\left(\frac{\overline{aw}+\overline{a'w}}{2}\right) > \tan(\overline{bw})\cos(C) \]\[
2\tan\left(\frac{\overline{aw}+\overline{a'w}}{2}\right) > 2\tan(\overline{bw})\cos(C)\]\[
2\cos(\overline{bw})\sin\left(\frac{\overline{aw}+\overline{a'w}}{2}\right) > 2\sin(\overline{bw})\cos\left(\frac{\overline{aw}+\overline{a'w}}{2}\right)\cos(C)\]

We multiply both sides by $\sin\left(\frac{\overline{a'w}-\overline{aw}}{2}\right)$, reversing the inequality since this value is negative, and apply angle difference formulas to obtain

\[
\cos(\overline{bw})\left(\cos(\overline{aw}) - \cos(\overline{a'w})\right) < \sin(\overline{bw})\left(\sin(\overline{a'w}) - \sin(\overline{aw})\right)\cos(C)\]\[
\cos(\overline{bw})\cos(\overline{aw})+\sin(\overline{bw})\sin(\overline{aw})\cos(C) < \cos(\overline{bw})\cos(\overline{a'w})+\sin(\overline{bw})\sin(\overline{a'w})\cos(C)\]

Applying the spherical law of cosines to the triangles with vertices $ABW$ and $A'BW$, we end up with

\[
\cos(\overline{ab}) < \cos(\overline{a'b})\]\[
\overline{ab} > \overline{a'b}\]

\end{proof}

Let $v = \projM(e_n)$.  Then $v$ is a vertex of $L$.  Call the span of the other $n-1$ vertices $L'$.  Define $\Psi\colon L - \{v\} \rightarrow L'$ such that $\Psi(q)$ is the point of $L'$ obtained by extending the geodesic segment $[v, q]$ to $L'$.

\begin{proposition}\label{P:projProps}
\begin{enumerate}
\item $\Psi$ is induced by orthogonal projection of $E$ to the subspace spanned by $\{e_1,...,e_{n-1}\}$.

\item Up to reparameterization, $\Psi$ sends geodesic segments to geodesic segments.

\item $\Psi$ sends $W$-directed paths to $\left(W - \{v\}\right)$-directed paths.

\item If $d(v, q) = d(v,q')$, then $d(q,q') \leq d\left(\Psi(q),\Psi(q')\right)$.
\end{enumerate}
\end{proposition}

\begin{proof}  Note that an injective path in $L$ is a geodesic segment (up to reparameterization) if and only if it lies in $\overline{P} \cap L$ for some 2-plane $P \subset E$ which passes through the origin.  Let $q \in L - \{v\}$, let $x$ be any vector such that $\projM(x) = q$, and let $x'$ be such that $\projM(x') = \Psi(q)$.  Since $v$, $q$, and $\Psi(q)$ lie on the same geodesic, this means that $x$ must lie in the plane spanned by $e_n$ and $x'$.  But since $\projM(x') \in L'$, $x'$ lies in the span of $e_1,...,e_{n-1}$, so $x'$ and $e_n$ are orthogonal.

Since $x = ae_n + bx'$ for scalars $a$ and $b$, and $\Psi(\projM(x)) = \projM(x')$, we have that $\Psi$ is induced by orthogonal projection to the subspace spanned by $e_1,...,e_{n-1}$, proving (2).  This projection is a linear map, thus it sends line segments to line segments, and since geodesic segments on $L$ are projections of line segments in $E$, we have proved (1).

For $q \in L$ and $w$ a vertex, $\Psi$ sends the geodesic segment $[q,w]$ to the geodesic segment $[\Psi(q),\Psi(w)]$.  If $w \neq v$, then $w \in L'$ and so $\Psi(w) = w$.  Thus $\Psi$ sends the segment $[q,w]$ to $[\Psi(q),w]$, and sends initial subsegments of the former to initial subsegments of the latter.  If $w = v$, a subsegment of $[q,w]$ is sent to a single point, which proves (3).

By (1), we have that for all $q, q' \in L$, $d(\Psi(q),v) = d(\Psi(q'),v) = \frac{\pi}{2}$.  Geodesic rays $\alpha(t), \alpha'(t)$ emitting from a common basepoint in a sphere have the property that $d\left(\alpha(t),\alpha'(t)\right) < d\left(\alpha(t'),\alpha'(t')\right)$ for $0 \leq t < t' \leq \frac{\pi}{2}$.  This proves (4).

\end{proof}

We require one final lemma:

\begin{lemma}\label{L:inPlace}
For any $\epsilon > 0$, there is an $R>0$ such that, if $x$ is a vector with positive coordinates and $\lVert x \rVert > R$, there is a walk $c\colon [0,\infty)\rightarrow E$ over $V$, starting at $x$, such that $\projM(c) \subset B_\epsilon(x)$.
\end{lemma}

\begin{proof} We will proceed by induction on the dimension of $L$.  If $L$ is 1-dimensional, then $B_\epsilon(\projM(x))$ is an interval.  If $\lVert v \rVert > \sin^{-1}\left(\frac{2}{\epsilon}\right)$ then Proposition \ref{P:projectSeg} guarantees that these segments all have length $< \frac{\epsilon}{2}$, thus we can always choose an $e_i$ such that $\projM([v,v+e_i]) \subset B_\epsilon(\projM(x))$.  Thus, setting $R = \sin^{-1}\left(\frac{2}{\epsilon}\right)$ gives the result.

For higher dimensions, set $v = e_n$, and let $L'$ and $\Psi$ be as in Proposition \ref{P:projProps}.  Let $\rho \colon E \rightarrow E$ be orthogonal projection to the subspace spanned by $e_1,...,e_{n-1}$.  By \ref{P:projProps}(1), for every vector $x$, $\Psi(\projM(x)) = \projM(\rho(x))$.  Note that $\lVert x \rVert \geq \lVert \rho(x) \rVert$.

By the inductive hypothesis, there is an $R$ such that if $\lVert \rho(x) \rVert > R$, then there is a walk in $\{e_1,...,e_{n-1}\}$, starting from $\rho(x)$, whose projection to $L'$ lies within $\frac{\epsilon}{3}$ of $\Psi(\projM(x))$.  If we use the same sequence of vectors starting from $x$ instead of $\rho(x)$ we obtain a walk $c$. By Proposition \ref{P:projProps} (1) and (4), we have that $\Psi(\projM(c)) \subset B_\frac{\epsilon}{3}\bigl(\Psi(\projM(x))\bigr)$.

Note that since $c$ is a walk in $e_1,...,e_{n-1}$, we have that for all $t$, $d\bigl(\projM(c(t)),v\bigr) \geq d\bigl(\projM(x),v\bigr)$.  Assume $R$ is chosen large enough so that projections of segments $[v,v+e_i]$ have length no more than $\frac{\epsilon}{3}$.  We will create another walk $c'$ by inserting copies of $e_n$ into the sequence of vectors for $c$.

Let $a$ be the smallest integer such that $d\bigl(\projM(c(a)),v\bigr) - d\bigl(\projM(x),v\bigr) > \frac{\epsilon}{3}$.  We insert copies of $e_n$ at the $a^{th}$ position in the sequence for $c$.  We insert the minimal number $m$ of such vectors so that $d\bigl(\projM(c(a)+me_n)),v\bigr) \leq d\bigl(\projM(x),v\bigr)$.  Such an $m$ exists by Proposition \ref{P:closeVert}.  After inserting the copies of $e_n$, our sequence continues with the sequence of vectors from $c$.

This creates a new sequence of vectors, and we continue in this way, finding the minimal integer $a$ such that $d\bigl(\projM(c(a)),v\bigr) - d\bigl(\projM(x),v\bigr) > \frac{\epsilon}{3}$, and inserting $e_n$'s so that $d\bigl(\projM(c(a+m)),v\bigr) \leq d\bigl(\projM(x),v\bigr)$.  Continuing this process to infinity creates a new walk $c'$ with the property that, for all $t$, $\lvert d\bigl(\projM(c'(t)),v\bigr) - d\bigl(\projM(x),v\bigr)\rvert < \frac{2\epsilon}{3}$.

For any $t$, let $q_t$ be the point on the geodesic $[v,\Psi(\projM(x))]$ such that $d(q_t,v) = d\bigl(\projM(c'(t)),v\bigr)$. Note that since we have only inserted $e_n$'s, we have that $\Psi(\projM(c)) = \Psi(\projM(c'))$.  By Proposition \ref{P:projProps} (4) this implies that $d\bigl(\projM(c'(t)),q_t\bigr) < \frac{\epsilon}{3}$.  Since $d\bigl(\projM(x),q_t\bigr) < \frac{2\epsilon}{3}$, the triangle inequality implies that $\projM(c'(t))$ is always within $\epsilon$ of $x$.

This construction works as long as $\lVert \rho(x) \rVert > R$.  But the construction using the inductive hypothesis works via orthogonal projection to the span of any $(n-1)$ of our basis vectors.  Thus, by replacing $R$ with $2R$, we can guarantee that any $x$ with $\lVert x \rVert > R$ has some sufficiently large projection to allow the above construction to work.  This completes the proof.
\end{proof}

\begin{proof}[Proof of Lemma \ref{L:closeWalk}] Let $N$ be the number of directed segments in $\gamma$.  Let $w_i = \projM(e_i)$, so $W = \{w_1,...,w_n\}$.  Let $\delta = \frac{\epsilon_1}{4N}$.  Let $R_1 = \csc(\delta)$, let $R_2$ to be the value obtained from Lemma \ref{L:inPlace} which allows the projection of a walk to stay within $\frac{\epsilon_2}{2}$ of its starting point, and let $R = \max(R_1,R_2)$.

To construct our walk $c$, let $\gamma_1, \gamma_2,...$ be the $W$-directed segments of $\gamma$.  Let $w_{i_k}$ be the vertex that $\gamma_k$ moves towards, and let $T_k$ be the terminal point of $\gamma_k$.  We will construct our sequence of vectors $c$ by starting with the empty sequence and appending copies of $e_{i_k}$ to the sequence for each $\gamma_k$.  Denote by $\Sigma_k$ the sum of all vectors appended up to the $k^{th}$ step of this process, together with $x$.  To choose how many $e_{i_k}$'s we append at the $k^{th}$ step, we consider 3 cases:

\begin{enumerate}

\item If $T_k = w_{i_k}$, we append enough copies of $e_{i_k}$ so that $d\bigl(\projM(\Sigma_k),w_{i_k}\bigr) < \frac{\epsilon_1}{2}$.

\item If $d(T_k, w_{i_k}) > d\bigl(\projM(\Sigma_{k-1}),w_{i_k}\bigr)$, we append no copies of $e_{i_k}$.

\item Otherwise, we append the minimal number of $e_{i_k}$ required so that $d(T_k, w_{i_k}) > d\bigl(\projM(\Sigma_k), w_{i_k}\bigr).$

\end{enumerate}

Cases (1) and (3) are possible by Proposition \ref{P:closeVert}.

We claim that in each of the above cases, $$d\bigl(\projM(\Sigma_{k}), T_{k}\bigr) \leq \max\left(\frac{\epsilon_1}{2}, d\bigl(\projM(\Sigma_{k-1}),T_{k-1}\bigr) + \delta\right).$$

In case 1, the claim is clear since $T_k = w_{i_k}$.  In case 2, the claim holds by Lemma \ref{L:getsClose}.

In case 3, let $A_k \in [T_{k-1},w_{i_k}]$ and $B_k \in [\projM(\Sigma_{k-1}),w_{i_k}]$ be such that $d(A_k, w_{i_k}) = d(B_k,w_{i_k}) = \min\bigl(d(T_{k-1},w_{i_k}), d(\projM(\Sigma_{k-1}),w_{i_k})\bigr)$.  Let $P_k$ be the point on $[\projM(\Sigma_{k-1}),w_{i_k}]$ such that $d(P_k,w_{i_k}) = d(T_k,w_{i_k})$.  Then we have

\[
d\bigl(\projM(\Sigma_{k-1}),T_{k-1}\bigr) \geq d(A_k,B_k)\]\[
d(A_k,B_k) \geq d(T_k, P_k)\]

The first inequality holds by Lemma \ref{L:getsClose}.  The second holds because geodesics $\alpha(t)$ and $\alpha'(t)$ emitting from a common basepoint have increasing distance in $t$ for $t < \frac{\pi}{2}$, and $[w_{i_k},A_k]$ and $[w_{i_k},B_k]$ are such segments.

Since we added the minimal number of $e_{i_k}$'s to make $\projM(\Sigma_k)$ closer to $w_{i_k}$ than $T_k$ is, we have that $P_k \in \projM\left([\Sigma_k - e_{i_k},\Sigma_k]\right)$.  But since $\lVert \Sigma_k - e_{i_k} \rVert \geq R = \csc(\delta)$, by Proposition \ref{P:projectSeg} the length of this segment is no more than $\sin^{-1}\left(\frac{1}{\csc(\delta)}\right) = \delta$.  Thus by the triangle inequality we have

\[
d(T_k,P_k) \geq d(\projM(\Sigma_k),T_k) - \delta
\]

and the claim follows.

Thus, for each $k$, $\projM(\Sigma_k)$ gets no more than $\delta$ further from $T_k$.  But since $\delta \leq \frac{\epsilon_1}{4N}$, with $N$ the number of segments in $\gamma$, we have that $d(\projM(\Sigma_N),T_N) < \frac{3\epsilon_1}{4}$, since $d(\projM(\Sigma_0),T_0) < \frac{\epsilon_1}{2}$.  Thus, it only remains to show that we can append additional vectors causing the walk to end within $\epsilon_2$ of $T_N$, and to show that we can continue the walk arbitrarily far without leaving $B_{\epsilon_1}(T_N)$.

By Lemma \ref{L:lineApprox}, we can find a $W$-directed path $\beta$ in $L$ which starts at $\projM(\Sigma_N)$, ends within $\epsilon_2$ of $T_N$, and stays within $\epsilon_2$ of the geodesic segment $[\Sigma_N, T_N]$.  Without loss of generality we may assume $\epsilon_2 < \frac{\epsilon_1}{4}$.  Let $N'$ be the number of segments in $\beta$, let $\delta' = \frac{\epsilon_2}{4N'}$, and let $R' = \csc(\delta')$.

By Lemma \ref{L:inPlace}, we can append vectors to our sequence so that the projections stay within $\frac{\epsilon_2}{2}$ of $\projM(\Sigma_N)$, until we reach a total sum $\Sigma$ with $\lVert \Sigma \rVert > R'$.  We can then repeat the above construction to append a walk whose projection ends within $\epsilon_2$ of $T_N$ and stays within $\epsilon_2$ of $\beta$ and thus within $\epsilon_1$ of $T_N$.  This gives us our walk which has all the desired properties.
\end{proof}

\section{Acknowledgments}\label{S:Ack}
The author would like to thank Kim Ruane for suggesting the question addressed in this paper, Nan Li for assistance with Lemma \ref{L:getsClose}, and Steve Ferry for numerous helpful comments and advice.

\end{document}